\documentclass[11pt,a4paper]{amsproc}

\usepackage[usenames]{color}
\usepackage{amsmath, amssymb, amsthm, latexsym,nicefrac,setspace, todonotes, graphicx, mathtools}
\usepackage[pagebackref,colorlinks,linkcolor=red,citecolor=blue,urlcolor=blue,hypertexnames=true]{hyperref}
\usepackage{tikz}
\usepackage{listings}
\theoremstyle{plain}
\newtheorem{theorem}{Theorem}[section]
\newtheorem{question}[theorem]{Question}
\newtheorem{lemma}[theorem]{Lemma}
\newtheorem{corollary}[theorem]{Corollary}
\newtheorem{proposition}[theorem]{Proposition}

\theoremstyle{remark}
\newtheorem{remark}[theorem]{Remark}
\newtheorem{claim}[theorem]{Claim}

\def\Z{\mathbb Z}

\def\C{\mathbb{C}}

\def\N{\mathbb{N}}

\newcommand{\eps}{\varepsilon}

\DeclareMathOperator{\PSL}{PSL}
\DeclareMathOperator{\diag}{diag}

\author[V. Alekseev]{Vadim Alekseev}
\address{Vadim Alekseev, TU Dresden, Germany}
\email{vadim.alekseev@tu-dresden.de}

\author[A. Carderi]{Alessandro Carderi}
\address{Alessandro Carderi}
\email{alessandro.carderi@gmail.com}

\author[A. Thom]{Andreas Thom}
\address{Andreas Thom, TU Dresden, Germany}
\email{andreas.thom@tu-dresden.de}

\author[R. Tucker-Drob]{Robin Tucker-Drob}
\address{Robin Tucker-Drob, University of Florida, USA}
\email{r.tuckerdrob@ufl.edu}

\title[Discrete subgroups of full groups]{About discrete subgroups of full groups of measure preserving equivalence relations}

\begin{document}

\onehalfspace

\begin{abstract} In this note we study countable subgroups of the full group of a measure preserving equivalence relation. We provide various constraints on the group structure, the nature of the action, and on the measure of fixed point sets, that imply that the subgroup topology is not discrete. We mention various conjectures about discrete subgroups of full groups.
\end{abstract}

\maketitle

\tableofcontents

\section{Introduction}

In this note we continue a study that was started in \cite{andreasvadim}. It is our attempt to understand the structure of discrete subgroups of full groups of equivalence relations induced by a probability measure preserving action of a countable group.

Let $\Gamma$ be a countable (discrete) group and let $\Gamma \curvearrowright (X,\mu)$ be an ergodic measure preserving action. Let $G:=[\Gamma \curvearrowright (X,\mu)]$ be the full group of the associated measurable equivalence relation. For any $g \in G,$ we denote by $S(g)$ its support, i.e., $S(g):=\{x \in X \mid g.x \neq x\}.$ The group $G$ is endowed with a natural bi-invariant metric associated with the conjugation invariant length function $\ell(g):= \mu(S(g))$, i.e., $d(g,h):= \ell(gh^{-1})$. 
For any subgroup $\Lambda \leq G$, we set
$$\delta(\Lambda) := \inf \{ \mu (S(g)) \mid g\in \Lambda \setminus \{ 1_\Lambda \}  \}$$
and call it the {\it modulus of discreteness} of $\Lambda.$ We say that $\Lambda$ is $\delta$-discrete if $\delta(\Lambda) \geq \delta.$ Note that $\delta(\Lambda)=1$ if and only if the action of $\Lambda$ is essentially free. Moreover, $\delta(\Lambda)>0$ if and only if $\Lambda$ is a discrete subgroup of $G$, i.e., the subgroup topology is discrete.

We say that $\Gamma \curvearrowright (X,\mu)$ is {\it maximal among discrete actions with the same orbits} or just {\it maximal discrete} if there is no non-trivial discrete subgroup $\Lambda$ of $G$ that contains $\Gamma$ properly.
It is a natural question to wonder if $G$ contains maximal discrete subgroups at all or if a given discrete subgroup is itself maximal discrete or at least contained in a maximal discrete subgroup.

In order to state our main results, we have to say a few words about MIF groups. A group $\Gamma$ is said to be {\it mixed identity free} (or MIF for short) if there exists no non-trivial word $w \in \Gamma \ast \mathbb Z$, such that the associated word map $w \colon  \Gamma \to \Gamma$, that is given by evaluation of the variable, satisfies $w(g)=1_\Gamma$ for all $g \in \Gamma.$
There has been a lot of study on MIF groups in recent years, see for example \cite{bst, MR3556970,MR4193626} and the references therein. Natural examples of MIF groups include free groups and ${\rm PSL}_n(\mathbb Z)$. More generally, by results of Hull and Osin, acylindrically hyperbolic group with trivial finite radical \cite{MR3556970} as well as Zariski-dense subgroups of particular simple Lie groups, see for example \cite{MR728928}, are MIF.

We have three main results covering various natural problems in that general area:

\begin{theorem}\label{main1}
\label{sequence} Let $\Gamma$ be group and $\Gamma \curvearrowright (X,\mu)$ be a faithful p.m.p.\ action on a standard probability space, such that at least one of the following conditions is satisfied:
\begin{enumerate}
    \item The group $\Gamma$ is {\rm MIF}, the action is weakly mixing, and essentially free.
    \item The group $\Gamma$ is torsion-free and the action is mixing.
\end{enumerate}
If $g \in G$ is such that $\mu(S(g))<1/3$, then the group $\langle \Gamma,g \rangle$ is not discrete. In particular, $\Gamma$ is contained in a maximal discrete subgroup.
\end{theorem}

By a result of Jacobson \cite{MR4193626}, there exists an elementary amenable MIF group, so that the above theorem implies that the full group of the hyperfinite equivalence relation admits maximal discrete subgroups, see Remark \ref{jacobs}.

The second result is concerned discreteness of an arbitrary p.m.p.\ ergodic  action of a {\rm MIF} group. 

\begin{theorem}\label{main2}
Let $\Gamma$ be a {\rm MIF} group and let $\Gamma \curvearrowright (X,\mu )$ be a faithful and ergodic p.m.p.\ action.  If $\Gamma$ is discrete, then it is $1/2$-discrete.
\end{theorem}

Note that Theorem \ref{main1} is not a consequence of this result, since we do not know a priori whether $\langle \Gamma,g\rangle$ is {\rm MIF} or not.


Our third result is concerned with compact actions, where the situation is much easier.

\begin{theorem}\label{thm:compact}
Let $\Gamma$ be a {\rm MIF} group and let $\Gamma \curvearrowright (X,\mu )$ be a faithful, p.m.p., ergodic and compact action. If $\Gamma$ is discrete, then the action is essentially free.
\end{theorem}

The assumption {\rm MIF} is essential in Theorems \ref{main2} and \ref{thm:compact}, in fact, we show that for group $\Gamma$ acting essentially freely on its pro-finite completion, there exists an ascending sequence of groups containing $\Gamma$ as a finite index subgroup, whose union is dense, see Section \ref{profinite}. An example of a different kind based on restricted wreath products can be found in Remark \ref{rem:example2}.

The gist of the proofs is different to the arguments in \cite{andreasvadim}. This time the contraction principle is derived from Khintchine's inequality resp.\ the various mixing conditions that are assumed.
The rough idea of the proof is as follows. A more or less classical observation from the theory of permutation groups yields that
$$\mu(S([g,hgh^{-1}]) \leq 3 \mu(S(g) \cap hS(g)).$$ Now, ergodicity implies that for given $g \in \Gamma$, there exists $h \in \Gamma$, such that $$\mu(S(g) \cap hS(g)) \leq \mu(S(g))^2 + \varepsilon.$$ Thus, we can construct shorter and shorter elements, provided there exists $g \in \Gamma$ with $\mu(S(g))<1/3$. The remaining assumptions are needed to overcome the problem that the iterated commutators could in fact become trivial in $G$ and thus be not useful for the purpose of proving non-discreteness. This gives the idea of a proof of Theorem \ref{main1} and of Theorem \ref{main2} when $1/2$ is replaced by $1/3$ -- while the case $1/2$ needs a more sophisticated approach. These arguments and the proof of Theorem \ref{thm:compact} make use of a finer study of the support of the commutator, see Proposition \ref{prop:subset}. 

\vspace{.2cm}

We would like to mention that to the best of our knowledge, there is no known example of a discrete ergodic action of a MIF group, which is not essentially free. Thus, we put forward the following question:

\begin{question}\label{qYG}
Is every faithful, discrete, ergodic and p.m.p.\ action of a countable MIF group essentially free?
\end{question}
We heard this question in some form from Yair Glasner 15 years ago (for free groups), but maybe its origins go back even further. A positive answer amounts to replacing $1/2$ by $1$ in Theorem \ref{main2}. 
A related question, whose negative answer would be a consequence of a positive answer to the previous question, is the following:
\begin{question} \label{q2}
Does the full group of a hyperfinite equivalence relation contain a discrete free subgroup?
\end{question}
In \cite[Section 1.2]{MR3568978}, a negative answer to this question appeared as a conjecture attributed to the third-named author. Related to this, it might be that every discrete subgroup of the full group of the hyperfinite equivalence relation is amenable. Note, that it was proved that discrete free subgroups exist in the group of invertible elements of the mod-$p$ analogue of the hyperfinite II$_1$-factor, see \cite{MR3795479}. The analogous question for the unitary group of the actual hyperfinite II$_1$-factor is open.

\section{Contraction from equidistribution for mixing actions}

\label{contraction}

\subsection{Preliminaries on actions and mixed identities}
Let $\Gamma$ be a group. For any $g,h\in \Gamma$ define $[g,h] := ghg^{-1}h^{-1}$, so that $[g,h]^{-1}=[h,g]$. If $\Gamma\curvearrowright X$ is an action of $\Gamma$ on a set $X$ then for $g\in \Gamma$ define the sets
\[
S(g) \coloneqq \{ x \in X \mid gx \neq x \} \quad \text{and} \quad
F(g) \coloneqq \{ x \in X \mid gx=x \} = X\setminus S(g) .
\]
Then $S(g)=S(g^{-1})$ and $S(ghg^{-1})=gS(h)$ (and similarly for $F$ in place of $S$).

\begin{proposition}\label{prop:subset}
Let $\Gamma \curvearrowright X$ be an action of a group on a set $X$. Let $g,h\in \Gamma$ and let $A_{g,h}= S(g)\cap S(h)$. Then we have
\[
S([g,h]) \subseteq A_{g,h}\cup gA_{g,h}\cup hA_{g,h} 
\]
Moreover, if $\mu$ is a $\Gamma$-invariant probability measure on $X$ then
\[
\mu (S([g,h])) \leq 3\mu (A_{g,h}) - \mu (gA_{g,h}\cap A_{g,h}) - \mu (hA_{g,h}\cap A_{g,h}) .
\]
\end{proposition}

\begin{proof}
We begin with the containment. Suppose that $x\in S([g,h])$, i.e., suppose that $ghg^{-1}h^{-1}x\neq x$. It is clear that this implies that $x\in S(g)\cup S(h) = (S(g)\cap S(h))\cup (S(g)\setminus S(h)) \cup (S(h)\setminus S(g))$. If $x\in S(g)\cap S(h)$ then we are done, so we just need to consider the two cases (1) $x\in S(g)\setminus S(h)$, and (2) $x\in S(h)\setminus S(g)$. In case (1) we have $x\neq ghg^{-1}h^{-1}x=ghg^{-1}x$, so $x\in S(ghg^{-1})$, and hence $x \in S(g)\cap S(ghg^{-1})=g(S(g)\cap S(h))$, as desired. Since $S([g,h])=S([h,g])$, this case is symmetric to Case (1) but with the roles of $g$ and $h$ interchanged, and hence in this case we have $x\in h(S(g)\cap S(h))$.

\begin{figure}[h] \label{venn4}
\includegraphics[width=220pt]{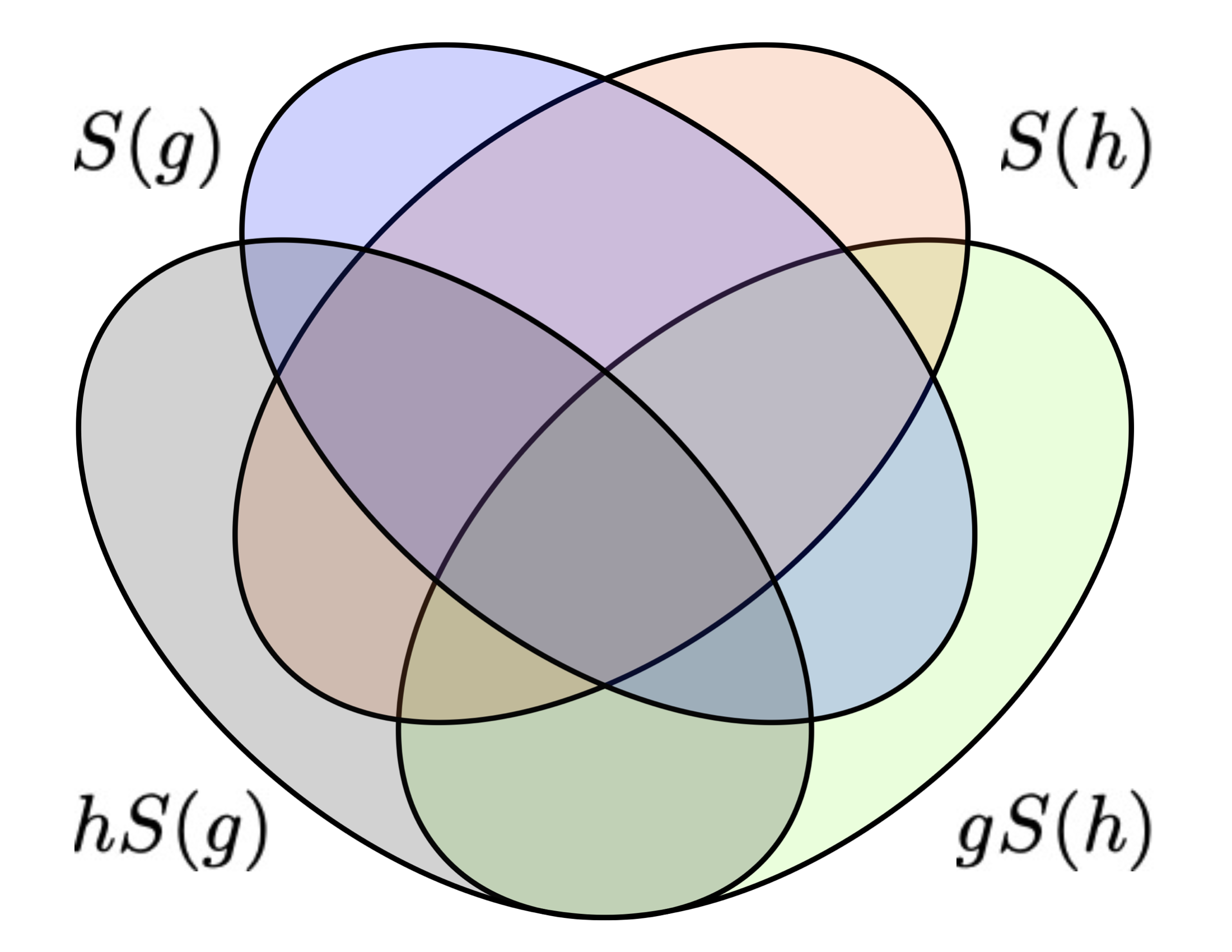}
\caption{Venn diagram of four sets}
\end{figure}

The inequality now follows from applying the inclusion-exclusion principle and noting that two of the terms cancel out, see Figure \ref{venn4}.
\end{proof}

\begin{lemma}[\cite{MR1774423}]\label{lem:recurrence}
Let $1\geq c>0$ and $1\geq \epsilon >0$ be given and let $n$ be an integer with $n>\frac{1-c}{c\epsilon} + 1$. Let $(Y, \nu )$ be a probability space, let $a\geq c$, and suppose that $A_0,\dots , A_{n-1}$ are measurable subsets of $Y$ each having measure $a$. Then there exist distinct $0\leq i,j < n$ with $\nu(A_i\cap A_j)> a^2(1-\epsilon )$.

In particular, if $T:(Y,\nu )\rightarrow (Y,\nu )$ is an invertible measure preserving transformation and $A\subseteq Y$ is measurable with $\nu (A)\geq c$, then there exists some $1\leq j<n$ such that $\nu (T^jA\cap A)>\nu (A)^2(1-\epsilon )$.
\end{lemma}

\begin{proof}
Assume that $\nu (A_i\cap A_j)\leq a^2(1-\epsilon )$ for all distinct $i,j$, and we will show that $n\leq \frac{1-a}{a\epsilon} + 1$, which will give a contradiction since $\frac{1-a}{a\epsilon}+1\leq \frac{1-c}{c\epsilon}+1$. By Cauchy-Schwarz we have
\begin{align*}
n^2a^2 = \big( \int \sum _{i<n}1_{A_i}\, d\nu \big) ^2 \leq \int \big( \sum _{i<n}1_{A_i}\big) ^2 \, d\nu  &= \sum _{i<n}\nu (A_i) + \sum _{i<n}\sum _{j\neq i} \nu (A_i\cap A_j) \\
&\leq na + n(n-1)a^2(1-\epsilon),
\end{align*}
and solving for $n$ shows $n\leq \frac{1-a}{a\epsilon} + 1$.
\end{proof}

We need the following proposition:
\begin{proposition} \label{prop:mif}
Let $\Gamma$ be a group and suppose there exists $w_1,\dots,w_k \in \Gamma \ast \mathbb Z$, such that for all $g \in \Gamma$ there exists $1 \leq i \leq k$, such that $w_i(g)=1_{\Gamma}.$ Then, $\Gamma$ satisfies a non-trivial mixed identity.
\end{proposition}
\begin{proof}
This is a standard exercise using iterated commutators. See for example the proof of \cite[Lemma 2.2]{MR3451381}.
\end{proof}

\subsection{Weakly mixing actions}
Let's recall that a subset $T \subset \Gamma$ is called {\it syndetic}, if there exists a finite subset $S \subset \Gamma$ with $TS=\Gamma$. Recall Khintchine's lemma \cite{MR1556883} from 1934, which is a straightforward quantitative form of Poincar\'e's recurrence theorem.
\begin{lemma}[Khintchine] \label{MR1556883}
Let $\Gamma$ be a group and $\Gamma \curvearrowright (X,\mu)$ be p.m.p.\ ergodic action on a standard probability space. Let $A,B \subset X$ be measurable subsets. For every $\varepsilon>0$, there exists $g \in \Gamma$, such that
$$\mu(A)\mu(B) - \varepsilon < \mu(gA \cap B).$$
Moreover, the set $T:=\{g \in \Gamma \mid \mu(A)\mu(B) - \varepsilon < \mu(gA \cap B) \}$ is syndetic.
\end{lemma}

It is easy to see that a similar lemma holds when studying the bound $\mu(gA \cap B) < \mu(A)\mu(B) + \varepsilon$, i.e.\ the set of $g\in \Gamma$ for which it holds is syndetic. It is worth noting that in general, it is not possible to have both bounds at the same time -- just assuming ergodicity. However, assuming that the action is weakly mixing, we have the following result:

\begin{theorem}[Bergelson-Rosenblatt, \cite{MR921351}] \label{equid}
Let $\Gamma$ be a group and $\Gamma \curvearrowright (X,\mu)$ be p.m.p.\ weakly mixing action on a standard probability space. Let $A,B \subset X$ be measurable subsets. For every $\varepsilon>0$, the set 
$$T:=\{g \in \Gamma \mid |\mu(A)\mu(B) -  \mu(gA \cap B)|< \eps \}$$ is syndetic.
\end{theorem}
\begin{proof}
Apply \cite[Corollary 1.5]{MR921351} to the function $(1-\mu(A))\chi_A - \mu(A)\chi_{A^c}$.
\end{proof}

\begin{proof}[Proof of Theorem \ref{main1}(1):]
We denote the generator of $\mathbb Z$ by $t$. Let $\ell := \mu(S(g))$ and $\delta \in (0,1/3 - \ell(g))$. We define $\alpha_n:= (\ell- \delta)^{2^n}$ and $\gamma_n:= (3 \cdot\ell + \delta)^{2^n}/3$. We will define a sequence of elements $(w_n)_n$ in $\Gamma \ast \mathbb Z$ and denote the natural image of $w_n$ in $\langle \Gamma,g\rangle$ by $v_n:=w_n(g)$. Our construction is such that $\alpha_n < \mu(S(v_n)) < \gamma_n$ for all $n \in \mathbb N.$ In particular, $(v_n)_n$ is a sequence that converges to $1_X$ and is not equal to $1_X$ for all $n \in \mathbb N.$ It follows that $\langle \Gamma,g\rangle$ is not a discrete subgroup of $G$.

We start the inductive construction with $w_0=t$ and $v_0=g$, which clearly satisfies the claim. Let's assume that we constructed $w_n \in \Gamma \ast \mathbb Z$ already. By inductive assumption, $v_n \in \Gamma$ satisfies $\alpha:=\mu(S(v_n)) \in(\alpha_n, \gamma_n).$

We set $A := S(v_n)$ and 
$$\varepsilon:= \min\{\alpha(\gamma_n - \alpha), \alpha^2 - \alpha_n^2 \}>0.$$ Then, by Theorem \ref{equid}, the set 
$T:=\{h \in \Gamma \mid |\mu(A)^2 -  \mu(hA \cap A)|< \eps \}$ is syndetic. Let's consider some $h \in T.$ By Proposition \ref{prop:subset}, we have 
$$\mu(\{x \mid [hv_nh^{-1},v_n]x \neq x\}) < 3 (\alpha^2 + \varepsilon).$$ By our choice of $\varepsilon$, we get $$3 (\alpha^2 + \varepsilon) \leq 3 \alpha^2 + 3\alpha(\gamma_n - \alpha) = 3 \alpha \gamma_n < 3 \gamma_n^2 = \gamma_{n+1}.$$

We would like to set $$w_{n+1}(t) := [hw_n(t)h^{-1},w_n(t)]$$ and obtain $v_{n+1} = [hv_nh^{-1},v_n]$.
However, it remains to ensure the lower bound on $\mu(S(v_{n+1}))$. We claim that if $\Gamma$ is ${\rm MIF}$, then there exists a choice for $h \in T$ with $$\mu(S(v_{n+1}))=\mu(S([hv_nh^{-1},v_n]) \geq \mu(S(v_n))^2 - \eps.$$

Let's assume for a moment that $g \in G$ had a finite decomposition. It follows that $v_n$ has a finite decomposition, say using group elements from a finite set $E \subset \Gamma.$ Moreover, there exists a finite set of elements $F \subset \Gamma$, such that $\cup_{s \in F} Ts^{-1} = \Gamma$. Let now $q \in \Gamma$ be arbitrary and $s \in F$ such that $h:=qs \in T$.
The element $[hv_nh^{-1},v_n]$ acts on $x \in X$ as $he^{-1}_1h^{-1}e_2^{-1}he_3h^{-1}e_4$, where $e_1,e_2,e_3,e_4 \in E \cup \{e\}$. If $x \in hA \cap A$, then $e_3,e_4 \neq e$ and the word
$w(t):=te^{-1}_1t^{-1}e_2^{-1}te_3t^{-1}e_4 \in \Gamma \ast \mathbb Z$ is not conjugate to an element of $\Gamma$. Now, if the complement of the support of $[hv_nh^{-1},v_n]$ in $hA \cap A$ has positive measure, this element acts trivially on a set of positive measure, and hence $he^{-1}_1h^{-1}e_2^{-1}he_3h^{-1}e_4=1_\Gamma$ or in other words $w(q)=1_\Gamma$ for
$$w(t) = tse^{-1}_1(ts)^{-1}e_2^{-1}tse_3(ts)^{-1}e_4$$ since the action is assumed to be essentially free. Let us assume that this happens for all $h \in T.$ Then, each $q \in \Gamma$ satisfies a mixed identity with constants from $E \cup F$ as above. Since there are only finitely many such identities, Proposition \ref{prop:mif} implies that $\Gamma$ is not MIF.

Thus, there exists $h \in T$ with 
$$\mu(S(w_{n+1}(h))) \geq \mu(hA \cap A) > \alpha^2 - \varepsilon \geq \alpha^2_n = \alpha_{n+1}.$$
This finishes the proof in case where $g$ has a finite decomposition and $\Gamma$ is {\rm MIF}. Note that it proves a uniform lower bounds on $\mu(S(v_n))$ depending only on $\mu(S(g))$ and not on the complexity of the decomposition of $g$. 
Using this, we can now approximate an arbitrary $g \in G$ with $\ell(g)<1/3$ within $\delta'>0$ by an element $g' \in G$ of finite decomposition in order to obtain a word $w_n(t) \in \Gamma \ast \mathbb Z$ of controlled length and uniform bounds
$$\alpha_n < \mu(S(w_n(g')) < \gamma_n.$$ Taking $\delta'>0$ small enough shows that arbitrarily small non-trivial elements also exist in $\langle \Gamma,g \rangle.$ This finishes the proof.
\end{proof}

We would like to record another consequence of Khintchine's lemma to the structure of non-discrete subgroups of full groups.

\begin{proposition}
\label{dense}
Let $\Gamma \curvearrowright (X,\mu)$ be an ergodic action and let $\Lambda \subset [\Gamma \curvearrowright (X,\mu)]$ be a non-discrete subgroup containing $\Gamma$. Then,
$\{\mu(S(g)) \mid g \in \Lambda \} \subset [0,1]$ is dense in $[0,1]$.
\end{proposition}
\begin{proof} Suppose that there is a $\delta$-gap in $\{\mu(S(g)) \mid g \in \Lambda \} \subset [0,1]$, i.e., there exists $\alpha \in (0,1)$ with $(\alpha,\alpha + \delta) \cap \{\mu(S(g)) \mid g \in \Lambda \} = \varnothing$ and $\alpha \leq 1 - \delta$. Let us also assume that $\alpha$ lies in the closure of $\{\mu(S(g)) \mid g \in \Lambda \}.$ 
If $\Lambda$ is not discrete, there is an element $g \in \Lambda$ of non-trivial support of size less than $\delta$. Let $h \in \Lambda$ with $\mu(S(h)) \in [\alpha-\delta \mu(S(g))/3,\alpha].$
Now, using Khintchine's inequality from above the overlap $\mu(S(h) \cap tS(g))$ can be made as small as $\mu(S(g)) \mu(S(h)) + \delta \mu(S(g))/3$, so that
$htgt^{-1}$ has support at most
$$\mu(S(h)) + \mu(S(g)) < \alpha + \delta$$ and at least
\begin{eqnarray*}
&&\mu(S(h)) + \mu(S(g)) - \mu(S(h) \cap tS(g))\\
&\geq& \alpha - \delta \mu(S(g))/3 + \mu(S(g)) - (\mu(S(g)) \mu(S(h)) + \delta \mu(S(g))/3)\\
&\geq& \alpha - \delta \mu(S(g))/3 + \mu(S(g)) - \mu(S(g)) (1 - \delta) - \delta \mu(S(g))/3\\
&=&\alpha + \delta \mu(S(g))/3.
\end{eqnarray*}
This is a contradiction to the assumption that there was no value in the interval $(\alpha,\alpha + \delta).$ This finishes the proof.
\end{proof}

\subsection{Mixing actions}

We recall that a p.m.p.\ action of $\Gamma$ on $(X,\mu)$ is \textit{mixing} if for every measurable subset $A,B\subseteq X$ and $\varepsilon>0$, there is a finite set $F\subset \Gamma$ such that for every $\gamma\notin F$ we have \[\left|\mu(\gamma A\cap B)-\mu(A)\mu(B)\right|<\varepsilon.\]
For this we follow  \cite[Definition 2.11]{MR3616077} -- note that this notion is sometimes called {\it strong mixing}. Let us recall that factor of weakly mixing actions are weakly mixing.

\begin{lemma}
  Let $\Gamma$ be a countable group and let $\Lambda\leq \Gamma$ be a finite index subgroup. Consider a weakly mixing action $\Gamma\curvearrowright X$. Then the restriction of the action to $\Lambda$ is ergodic.
\end{lemma}
\begin{proof}
  Since $\Lambda$ contains a finite index normal subgroup, we can assume that $\Lambda$ is normal. Let us denote by $\mathcal A$ the $\sigma$-algebra of subsets fixed by $\Lambda$. Note that since $\Lambda$ is normal in $\Gamma$, the algebra $\mathcal A$ is $\Gamma$-invariant. We hence obtain a $\Gamma$-factor on which $\Lambda$ acts trivially.  Since this factor is also weakly mixing, the only possibility is that it is the trivial factor, that is if $\mathcal A$ is the trivial $\sigma$-algebra.
\end{proof}

\begin{lemma}\label{lem: if T commutes then support}
  Consider a weakly mixing action $\mathbb Z\curvearrowright X$ and let $g$ be a measure preserving transformation of $(X,\mu)$. Suppose that there is a positive measure subset $A\subseteq X$ and $z \in \mathbb Z\setminus\{0\}$ such that $gx=z x$ for all $x\in A$. If there is $k\in\mathbb Z \setminus \{0\}$ such that for every $n\in\mathbb Z$, $[g,z^{nk}g z^{-nk}]$ is trivial, then $S(g)$ has full measure. 
\end{lemma}
\begin{proof}
Given $u,v\in\mathrm{Aut}(X,\mu)$ such that $[u,v]$ is trivial, then $S(v)$ is $u$-invariant. Indeed, if $[u,v]$ is trivial, then for every $x\in X$, $uvu^{-1}x=vx$, that is $uF(v)=F(v)$ and hence $S(v)$ is $u$-invariant.

Fix $k\in\mathbb Z \setminus \{0\}$, we will show that if for every $n\in\mathbb Z$, the set $z^{nk}S(g)$ is invariant under $g$, then $g$ has full support. Let $\Lambda$ be the group generated by $z^k$. 
Denote by $\mathcal A$ the $\sigma$-algebra generated by $\{\lambda S(g);\ \lambda\in\Lambda\}$. Clearly $\mathcal A$ is $\Lambda$-invariant and $g$ acts trivially on it. Consider the $\Lambda$-factor $\pi\colon (X,\mu)\rightarrow (Y,\nu)$ associated to it. Remark that $g$ induces the trivial action on $(Y,\nu)$, that is $\pi(gx)=\pi(x)$ for every $x\in X$. In particular, we have that $\pi(z x)=\pi(x)$ for every $x\in A$. Since by the previous lemma the action of $\Lambda$ is ergodic, for almost every $x\in X$, there is $\lambda\in \Lambda$ such that $\lambda x\in A$. Then \[\pi(z x)=\pi(\lambda^{-1}z \lambda x)=\lambda^{-1}\pi(z \lambda x)=\lambda^{-1}\pi(\lambda x)=\pi(x)\]
that is, $\pi(z x)=\pi(x)$ for almost every $x\in X$. This implies that the action of $\Lambda$ on $Y$ is trivial. However since the action of $\Lambda$ on $X$ is ergodic, this can only happens when $Y$ is the one point space, that is when $S(g)$ has full measure. 
\end{proof}


\begin{proposition}
  Let $\Gamma$ be a countable torsion free group and consider a p.m.p.\ mixing action of $\Gamma$ on $(X,\mu)$. Then for every $\varepsilon>0$ and $g\in[\Gamma\curvearrowright (X,\mu)]$ non-trivial and not of full support, there is $h\in\Gamma$ such that $[g,h gh^{-1}]$ is non-trivial and \[\mu(S([g,h gh^{-1}]))<3\mu(S(g))^2+\varepsilon.\]
\end{proposition}
\begin{proof}
 There is $\lambda\in\Gamma$ and a positive measure subset $A\subseteq X$ such that $gx=\lambda x$ for every $x\in A$. 
Since the action of $\Lambda:= \langle \lambda \rangle$ is mixing, for every $\varepsilon>0$ there is $n_0\in\mathbb N$ such that for every $n\geq n_0$ we have that 
 \[\left|\mu(S(g)\cap \lambda^nS(g))-\mu(S(g))^2\right|<\frac\varepsilon 3.\]
 Proposition \ref{prop:subset} then implies that 
 \[\mu(S([g,\lambda^n g\lambda^{-n}]))<3\mu(S(g)\cap \lambda^nS(g))<3\mu(S(g))^2+\varepsilon.\]
 
 Since $g$ is not of full support, we apply Lemma \ref{lem: if T commutes then support} to obtain that there is $k\in\mathbb Z\setminus \{0\}$ such that $[g,\lambda^{nk}g\lambda^{-nk}]$ is not trivial. Thus, we can set $h:=\lambda^{nk}$ for any such $k \in \mathbb Z \setminus \{0\}$.
\end{proof}

\begin{proof}[Proof of Theorem \ref{main1}(2)] This is now immediate from the previous proposition. We define a sequence $(g_n)_n$ with $g_0=g$ and $g_{n+1} := [g_n,hg_nh^{-1}]$ for suitable $h \in \Gamma.$ If $\mu(S(g_0))<1/3,$ then the sequence $(g_n)_n$ consists of non-trivial elements and converges to the identity in the full group.
\end{proof} 

\begin{remark}
  Let $\Lambda$ be any countable group and set $\Gamma:=\Lambda\times \mathbb Z$. Let us denote by $z$ a generator of $\mathbb Z$.  
  Let us consider the Bernoulli actions of $\Lambda$ on $(\{0,1\}^\Lambda,\nu_\Lambda)$ and of $\Gamma$ on $(\{0,1\}^\Gamma,\nu_\Gamma)$. Since $\Gamma/\mathbb Z=\Lambda$, we will sometimes consider the former action as a non-faithful action of $\Gamma$ and since $\Lambda\leq \Gamma$ we will consider the latter also as a $\Lambda$ action. 
  Consider now 
 the space \[X:=\{0,1\}^\Lambda\times\{0,1\}^{\Gamma}=\{0,1\}^{\Lambda\sqcup \Gamma}\]
  equiped with the product measure $\mu:=\nu_\Lambda\times \nu_\Gamma$. Then $\Gamma$ acts on $(X,\mu)$ diagonally and the action preserves the measure. 
  
  We remark that the action of $\Lambda$ on $X$ is a generalized Bernoulli shift and that the action of  action of $\Lambda$ on $\Lambda\sqcup \Gamma$ is free. Moreover the product action of $\Gamma$ on $X\times X$ is again a generalized Bernoulli shift and hence it is ergodic. Therefore the action of $\Gamma$ on $X$ is weakly mixing, see \cite[Theorem 2.25]{MR3616077}. 
  
  Let $A\subset \{0,1\}^\Lambda$ be a positive measure subset and set $B:=A\times \{0,1\}^{\Gamma}$. Clearly $\mu(B)=\nu_\Lambda(A)$ and hence $B\subseteq X$ has positive measure. Remark also that $B$ is invariant by the $\mathbb Z$-action and denote by $g \in [\Gamma\curvearrowright (X,\mu)]$ the restriction of the generator $z$ of $\mathbb Z$ on the $\mathbb Z$-invariant set $B$. 
  
  Fix now $h=(\lambda,z^k)\in\Gamma$. Then for every $x\in \lambda A$, we have  
  \[h gh^{-1}(x,y)=h g(h^{-1}x,z^ky)=h (\lambda^{-1}x,z^{k+1}y)=(x,zy)\]
  and that if $x\notin \lambda A$, then $\lambda^{-1} x\notin A$ and hence $h gh^{-1}(x,y)=(x,y)$. That is, the element $h gh^{-1}$ is the restriction of the generator $z$ of $\mathbb Z$ on the $\mathbb Z$-invariant set $h B=\lambda A\times \{0,1\}^\Gamma$. In particular, $g$ and $h gh^{-1}$ commute. 
  
  Therefore the strategy we use in our results cannot be applied for this action. Observe that the action of $\Gamma$ is weakly mixing and not mixing and that $\Gamma$ is not MIF so neither of our two results apply. 
  
\end{remark}

\section{Ergodic actions of MIF groups} 

The aim of this section is to prove Theorem \ref{main2}. We start with some preparations. Let $1>a_0>0$. In the proof of Lemma \ref{lem:MIFrecurrence} below we will use some properties of the family of sequences $(a_{m,\delta} )_{m=0}^{\infty}$ defined for $\delta \in [0,1)$ by
\begin{align*}
a_{0,\delta} &=a_0 \\
\text{and } \ a_{m+1,\delta}&=\frac{1}{2-a_{m,\delta}(1-\delta )}.
\end{align*}
We write $a_m$ for $a_{m,0}$.

\begin{proposition}\label{prop:amdelta}
Let $a_{m,\delta}$ and $a_m$ be defined as above. Then for all $m\geq 0$ we have:
\begin{enumerate}
\item $a_m=1- \frac{1-a_0}{m(1-a_0)+1}$, and hence $\lim _{m\rightarrow\infty} a_m = 1$.
\item $0<a_{m,\delta _1}\leq a_{m,\delta _0} <1$ whenever $1> \delta _1\geq \delta _0\geq 0$.
\item $\lim _{\delta \rightarrow 0^{+}} a_{m,\delta}=a_m$.
\item If $N_0\geq 0$ is an integer and $(b_m)_{m=0}^{N_0}$ is a sequence with $1\geq b_0\geq a_0$ and $1\geq b_{m+1}\geq \frac{1}{2-b_m(1-\delta)}$ for all $0\leq m<N_0$, then $b_m\geq a_{m,\delta}$ for all $0\leq m\leq N_0$.
\item If $a_0\leq \frac{1}{1+\sqrt{\delta}}$ then $a_{0,\delta}\leq a_{1,\delta}\leq a_{2,\delta}\leq \cdots$, and $\lim _{m\rightarrow\infty}a_{m,\delta} = \frac{1}{1+\sqrt{\delta}}$.
\end{enumerate}
\end{proposition}

\begin{proof}
Items (1), (2), (3), and (4) all follow by induction on $m$. For (5), assume that $a_0\leq \frac{1}{1+\sqrt{\delta}}$. Then induction on $m$ shows that $a_{m,\delta}\leq \frac{1}{1+\sqrt{\delta}}$ for all $m$, i.e., $\delta \leq \frac{(1-a_{m,\delta})^2}{a_{m,\delta}^2}$ for all $m$. We have $a_{m,\delta}\leq a_{m+1,\delta}$ if and only if $(1-\delta)a_{m,\delta}^2-2a_{m,\delta}+1\geq 0$, which is seen to hold for all $m$ using that $\delta \leq \frac{(1-a_{m,\delta})^2}{a_{m,\delta}^2}$. The sequence $(a_{m,\delta})_{m\geq 0}$ is monotone nondecreasing and its limit $L_{\delta}$ is bounded above by $\frac{1}{1+\sqrt{\delta}}$ and satisfies $L_{\delta}=\frac{1}{2-L_{\delta}(1-\delta)}$, and hence $L_{\delta} = \frac{1}{1+\sqrt{\delta}}$.
\end{proof}

\begin{lemma}\label{lem:MIFrecurrence}
Let $\Gamma\curvearrowright (X,\mu )$ be a p.m.p.\ action of a MIF group $\Gamma$, and let $g$ be a nonidentity element of $\Gamma$. Given $\epsilon >0$ and a natural number $N$, there exists an element $k\in \Gamma$ with $\mu (kS(g)\triangle S(g))<\epsilon$ such that $[k,g]$ has order at least $N$ and does not commute with $g$.
\end{lemma}

\begin{proof}
We may assume that $0<\mu (S(g))<1$. Let $a_0 := \mu (S(g))$ and let $a_m$ and $a_{m,\delta}$ be defined as above. Fix a natural number $N_0$ with $a_{N_0}>1 -\epsilon /2$. By parts (3) and (5) of Proposition \ref{prop:amdelta} we can find some $1>\delta >0$ such that $a_0=a_{0,\delta}\leq a_{1,\delta}\leq\cdots \leq a_{N_0,\delta}$ and $a_{N_0,\delta}>1-\epsilon /2$. Let $N_1>N$ be larger than $\frac{1-\mu (S(g))}{\mu (S(g))\delta} + 1$.

For $h\in \Gamma$ and positive integers $i_0,i_1, i_2,\dots $ we define $w_{i_0}(h):= h^{i_0}$, and $$w_{i_0,\dots ,i_{m-1},i_m}(h) := [w_{i_0,\dots , i_{m-1}}(h),g]^{i_m}$$ for each $m\geq 1$. Since $\Gamma$ is MIF, using Proposition \ref{prop:mif}, we may find some $h\in \Gamma$ such that for all $i_0,\dots , i_{N_0}\in \{ 1,\dots , N_1 \}$ the group element
\[
[w_{i_0,\dots , i_{N_0}}(h),g] = [[\cdots [[[h^{i_0},g]^{i_1},g]^{i_2},g]^{i_3}\cdots ,g]^{i_{N_0}},g]
\]
is nontrivial. In particular, given any choice of $i_0,\dots , i_{N_0-1}\in \{ 1,\dots , N_1 \}$, each of the group elements $h,[w_{i_0}(h),g], [w_{i_0,i_1}(h) ,g],\dots , [w_{i_0,\dots , i_{N_0-1}}(h),g]$ has order strictly greater than $N_1$.

For each non-null subset $B$ of $X$ let $\mu _B$ denote the normalized restriction of $\mu$ to $B$.  We will recursively define sets $X_0\supseteq X_1\supseteq \cdots \supseteq X_{N_0}\supseteq S(g)$, and $j_0,j_1,\dots ,j_{N_0 -1}\in \{ 1,\dots ,N_1 \}$ such that $X_m$ contains $S([w_{j_0,\dots , j_{m-1}}(h),g])$, and $$\mu _{X_{m+1}}(S(g)) >\frac{1}{2-\mu _{X_m}(S(g))(1-\delta )}$$ for all $m=0,\dots , N_0-1$.

We define $X_0=X$. By Lemma \ref{lem:recurrence} and our choice of $N_1$ we can find some $1\leq j_0\leq N_1$ such that $\mu _{X_0} (h^{j_0}S(g)\cap S(g)) > \mu _{X_0}(S(g))^2(1-\delta )$. Let $X_1=S(g)\cup h^{j_0}S(g)\subseteq X_0$, so that $\mu _{X_0}(X_1)<2\mu _{X_0}(S(g))-\mu _{X_0}(S(g))^2(1-\delta )$ and
\[
\mu _{X_1}(S(g)) > \frac{\mu (S(g))}{2\mu _{X_0}(S(g))-\mu _{X_0}(S(g))^2(1-\delta )} = \frac{1}{2-\mu _{X_0}(S(g))(1-\delta )} .
\]
The set $X_1$ contains both $S(g)$ and $h^{j_0}S(g)$, hence it contains $S([h^{j_0},g])=S([w_{j_0}(h),g])$.

Since $X_1$ contains $S([w_{j_0}(h),g])$ it is invariant under $[w_{j_0}(h),g]$. Therefore, we may apply Lemma \ref{lem:recurrence} to the transformation $[w_{j_0}(h),g]$ of $(X_1,\mu _{X_1})$ to find some $1\leq j_1\leq N_1$ such that $\mu _{X_1}([w_{j_0}(h),g]^{j_1}S(g)\cap S(g))>\mu _{X_1}(S(g))^2(1-\delta )$. Let $X_2 = S(g)\cup [w_{j_0}(h),g]^{j_1}S(g) \subseteq X_1$, so that $\mu _{X_1} (X_2) <2\mu _{X_2}(S(g))-\mu _{X_2}(S(g))^2(1-\delta )$ and
\begin{align*}
\mu _{X_2}(S(g)) = \frac{\mu _{X_1}(S(g))}{\mu _{X_1}(X_2)}&>\frac{\mu _{X_1}(S(g))}{2\mu _{X_1}(S(g))-\mu _{X_1}(S(g))^2(1-\delta)} \\
&= \frac{1}{2-\mu _{X_1}(S(g))(1-\delta )}.
\end{align*}
The set $X_2$ contains both $S(g)$ and $w_{j_0,j_1}(h)S(g)$ hence it contains $S([w_{j_0,j_1}(h),g])$. 

We continue this process: in general, if $2\leq m<N_0$ and we have already defined $j_0,\dots , j_{m-1}$, and $X_0\supseteq \cdots \supseteq X_{m}\supseteq S(g)$ with $X_{m}\supseteq S([w_{j_0,\dots , j_{m-1}}(h),g])$, we apply Lemma \ref{lem:recurrence} again to find some $1\leq j_m\leq N_1$ such that $$\mu _{X_m}(w_{j_0,\dots ,j_m}(h)S(g)\cap S(g)) > \mu _{X_m}(S(g))^2(1-\delta ).$$ We take $X_{m+1}= S(g)\cup w_{j_0,\dots ,j_m}(h)S(g)\subseteq X_m$ so that $\mu _{X_m}(X_{m+1}) < 2\mu _{X_m}(S(g))- \mu _{X_m}(S(g))^2(1-\delta )$ and
\begin{align*}
\mu _{X_{m+1}}(S(g)) = \frac{\mu _{X_m}(S(g))}{\mu _{X_m}(X_{m+1})} &>\frac{\mu _{X_m}(S(g))}{2\mu _{X_m}(S(g))-\mu _{X_m}(S(g))^2(1-\delta)} \\
&= \frac{1}{2-\mu _{X_m}(S(g))(1-\delta )} .
\end{align*}
Since $X_{m+1}$ contains both $S(g)$ and $w_{j_0,\dots ,j_m}(h)S(g)$, it contains $S([w_{j_0,\dots , j_{m}}(h),g])$.

We let $k:=w_{j_0,j_1,\dots , j_{N_0-1}}(h)$, so that both $k$ and $[k,g]$ have order at least $N_1>N$, and $X_{N_0}=S(g)\cup kS(g)$. It remains to show that our choice of $\delta$ implies that $\mu (kS(g)\triangle S(g))<\epsilon$. By (4) of Proposition \ref{prop:amdelta} we have $\mu _{X_{m}}(S(g))\geq a_{m,\delta}$ for all $0\leq m\leq N_0$, and hence $\mu _{X_{N_0}}(S(g))\geq a_{N_0,\delta}>1-\epsilon /2$. Since $X_{N_0}=S(g)\cup kS(g)$, this means that $\mu (S(g)\cup kS(g))< \mu (S(g)) + \epsilon /2$ and hence $\mu (S(g)\triangle kS(g))<\epsilon$.\qedhere
\end{proof}

\begin{proof}[Proof of Theorem \ref{main2}]
If $t=1$ then there is nothing to prove, so we may assume that $t<1$. Fix $1\geq \epsilon _0>0$, and find $0< \epsilon <\epsilon _0$ so small that $t(1-2\epsilon ) - t^2(1+\epsilon )^3 >\epsilon /\epsilon _0$.

By Lemma \ref{lem:recurrence} there exists an integer $N >0$ such that if $T$ is a measure preserving transformation on a probability space $(Y,\nu )$ and $C$ is a measurable subset of $Y$ with $\nu (C)\geq 1/6$ then there is some $1\leq i < N$ such that $\nu (T^iC\cap C)>\nu (C)^2(1-\epsilon ^2)$.

Let $g_0\in \Gamma\setminus \{ 1 \}$ be such that $t\leq \mu (S(g_0))<t(1 + \epsilon )$. Then, by Lemma \ref{lem:MIFrecurrence} we can find an element $g\in G$ of order greater than $N$ with $\mu (S(g))<t(1+\epsilon )$. By ergodicity of the action $\Gamma \curvearrowright (X,\mu )$, there is a syndetic subset $D$ of $\Gamma$ such that for all $k\in D$ we have $\mu (S(g)\cap kS(g)) < \mu (S(g))^2(1+\epsilon )<t^2(1+\epsilon )^3$. For a non-null $B\subseteq X$ let $\mu _B$ denote the normalized restriction of $\mu$ to $B$.
\begin{claim}\label{claim:MIFnoncommute}
There exists some $k\in D$ such that for all $1\leq i,j<N$ the group elements $g^i$ and $kg^jk^{-1}$ do not commute. Moreover, for every such $k$ we have both
\begin{align*}
\mu _{S(g)}(S(g)\cap S(kgk^{-1})) &\geq 1/6 \\
 \text{ and } \ \mu _{S(kgk^{-1})}(S(g)\cap S(kgk^{-1}))&\geq 1/6.
\end{align*}
\end{claim}

\begin{proof}[Proof of Claim \ref{claim:MIFnoncommute}] Fix a finite subset $F_D\subseteq G$ such that $DF_D=\Gamma$. Suppose no such $k$ exists as in the first statement in the claim. Then for every $h\in \Gamma$ there exists some $s\in F_D$ and some $1\leq i,j\leq N$ such that $h$ satisfies the nontrivial mixed identity $[g^i,hs^{-1}g^jsh^{-1}] =e$; the mixed identity is nontrivial since $g$ has order greater than $N$. Since there are only finitely many such triples $(s,i,j)$, Proposition \ref{prop:mif} shows that $\Gamma$ satisfies a nontrivial mixed identity, a contradiction.

For the moreover statement, given such a $k$, Proposition \ref{prop:subset} applied to the non-commuting elements $g$ and $kgk^{-1}$ shows that $$t\leq \mu (S([g,kgk^{-1}]))\leq 3\mu (S(g)\cap S(kgk^{-1})),$$ and hence $\frac{\mu (S(g)\cap S(kgk^{-1}))}{\mu (S(g))}\geq\frac{t}{3t(1+\epsilon )} \geq 1/6$. Since $S(kgk^{-1})=kS(g)$ we likewise have $$\frac{\mu (S(g)\cap S(kgk^{-1}))}{\mu (S(kgk^{-1}))}\geq 1/6.$$ This finishes the proof of Claim \ref{claim:MIFnoncommute}.
\end{proof}

Fix now $k\in D$ as in Claim \ref{claim:MIFnoncommute} and let $h=kgk^{-1}$. Since $S(g)$ is invariant under the cyclic group $\langle g\rangle$, by applying our choice of $N$ to the action $\langle g\rangle \curvearrowright (S(g), \mu _{S(g)})$ and the subset $S(g)\cap S(h)$ of $S(g)$, we obtain some $1\leq i<N$ such that $\mu _{S(g)}(g^i (S(g)\cap S(h))\cap S(g)\cap S(h)) \geq \mu _{S(g)}(S(g)\cap S(h))^2 (1- \epsilon ^2)$, i.e.,
\[
\mu (g^i (S(g)\cap S(h))\cap S(g)\cap S(h)) \geq \frac{\mu (S(g)\cap S(h))^2}{\mu (S(g))} (1- \epsilon ^2) \geq \frac{\mu (S(g)\cap S(h))^2}{t} (1-\epsilon ) .
\]
Likewise, applying our choice of $N$ to the action $\langle h\rangle \curvearrowright (S(h),\mu _{S(h)})$ and the subset $S(g)\cap S(h)$, we obtain some $1\leq j< N$ such that
\[
\mu (h^j (S(g)\cap S(h))\cap S(g)\cap S(h)) \geq \frac{\mu (S(g)\cap S(h))^2}{t} (1-\epsilon ) .
\]
Since $S(g^i)\subseteq S(g)$ and $t\leq \mu (S(g^i)) \leq \mu (S(g))\leq t(1+\epsilon )$, we have $\mu (S(g^i) \triangle S(g))< \epsilon t$. Likewise, $S(h^j)\subseteq S(h)$ and $\mu (S(h^j)\triangle S(h))<\epsilon t$. Therefore
\begin{align}
\label{eqn:gn} \mu (g^i (S(g^i)\cap S(h^j))\cap S(g^i)\cap S(h^j)) &\geq \frac{\mu (S(g^i)\cap S(h^j))^2}{t} (1-\epsilon ) - 4\epsilon t \\
\label{eqn:hm} \mu (h^j (S(g^i)\cap S(h^j))\cap S(g^i)\cap S(h^j)) &\geq \frac{\mu (S(g^i)\cap S(h^j))^2}{t} (1-\epsilon ) - 4\epsilon t
\end{align}
Let $A=S(g^i)\cap S(h^j)$. Observe that $\mu (A)\leq \mu (S(g)\cap S(h)) + 2\epsilon t < t^2(1+\epsilon )^3 + 2\epsilon t$, so by our choice of $\epsilon$ we have
\begin{equation}\label{eqn:tA}
\frac{\epsilon}{\epsilon _0} < t- \mu (A) .
\end{equation}
Our choice of $k$ ensures that $g^i$ and $h^j$ do not commute. Therefore, applying Proposition \ref{prop:subset} to $[g^i,h^j]$ and using \eqref{eqn:gn} and \eqref{eqn:hm}, we obtain
\begin{align*}
t \leq \mu (S([g^i,h^j]))  &\leq 3\mu (A) - 2\frac{\mu (A)^2}{t} (1-\epsilon ) + 8\epsilon t \\
&=3\mu (A) - 2\frac{\mu (A)^2}{t} +  2\frac{\mu (A)^2}{t}\epsilon + 8\epsilon t .
\end{align*}
Multiplying this inequality by $t$ (which by assumption is strictly positive) and rearranging gives $t^2 - 3\mu (A)t + 2\mu (A)^2 \leq 2\mu (A)^2\epsilon + 8\epsilon t^2 \leq 10\epsilon$ and hence
\begin{align*}
(t-2\mu (A)) (t-\mu (A)) & \leq 10\epsilon .
\end{align*}
If $t-2\mu (A)> 0$, then multiplying \eqref{eqn:tA} by $t-2\mu (A)$ shows that $(t-2\mu (A))\frac{\epsilon}{\epsilon _0} < 10 \epsilon$, and hence
\[
t\leq 2\mu (A) + 10\epsilon _0 \leq 2t^2(1+\epsilon )^3 + 2\epsilon t + 10 \epsilon _0 .
\]
If $t- 2\mu (A)\leq 0$, then this last inequality holds trivially. In either case, this last displayed inequality holds, so since $\epsilon _0>0$ was arbitrary, and since $\epsilon \rightarrow 0$ as $\epsilon _0\rightarrow 0$, it follows that $t\leq 2t^2$, and therefore $t\geq 1/2$.
\end{proof}

\begin{remark} \label{rem:example2}
Here is a first example, which shows that the assumption that $\Gamma$ is MIF is necessary in Theorem \ref{main2}. 
Given any infinite group $H$, consider a $\Z /2\Z$-vector space $V$ equipped with a basis $(\delta _h)_{h\in H}$ that is in bijection with $H$. 
The left translation action of $H$ permutes this basis, inducing an action by automorphisms on $V$, and we identify $V$ and $H$ naturally with subgroups of the associated semidirect product $V\rtimes  H$ (which is isomorphic to the restricted regular wreath product $(\Z /2\Z )\wr H$). 
Independently assign to each basis element a uniformly distributed label in $[0,1]$, and for $t\in [0,1]$ let $V_t$ be the (random) subspace of $V$ generated by those basis elements with label at most $t$. 
Then $V_t$ is an ergodic invariant random subgroup of $V\rtimes H$, hence by \cite{MR3165420}, $V_t$ is the stabilizer distribution of some ergodic p.m.p.\ action of $V\rtimes H$. 
Under this action, fixed point sets of group elements not lying in $V$ have measure zero, and the measure of the fixed point set of a vector $v\in V$ of the form $v=\sum 
 _{h\in Q}\delta  _h$ is exactly the probability that $v$ belongs to $V_t$, which is $t^{|Q|}$. Thus, when $t>1/2$ this gives an ergodic action of $V\rtimes  H$ for which $V\rtimes H$ is discrete, but not $1/2$-discrete.
\end{remark}

\section{Discrete groups and compact actions}

A p.m.p.\ action $\Gamma\curvearrowright (X,\mu )$ is {\it compact} if the image of $\Gamma$ in $\mathrm{Aut}(X,\mu )$ is precompact in the weak topology, i.e., the usual Polish group topology on $\mathrm{Aut}(X,\mu )$.

\subsection{Profinite actions} \label{profinite}

Let $\Gamma$ be a discrete group and let $(\Gamma_n)_n$ be a descending sequence of normal subgroups with $\Gamma=\Gamma_0$ and $\bigcap_n \Gamma_n = \{e\}$. We consider the corresponding profinite completion $\widehat \Gamma$ with its Haar measure $\mu_H$. We denote the closure of $\Gamma_n$ in $\widehat \Gamma$ by $\widehat \Gamma_n$.

Let $k:= [\Gamma:\Gamma_n]$ and let $g_1,\dots,g_k$ be a set of representatives of $\Gamma_n$-cosets, i.e. $\Gamma = \bigsqcup_{i=1}^k g_i\Gamma_n$. It follows that
\[
\widehat \Gamma = \bigsqcup_{i=1}^k  g_i \widehat \Gamma_n = [k] \times \widehat\Gamma_n.
\]
Let $G:=[\Gamma \curvearrowright (\widehat{\Gamma},\mu_H)]$, and let $T\in G$ be compatible with this decomposition, that is, $T$ acts by the left multiplication with $h_i \in \Gamma$ on $g_i\widehat \Gamma$. Then $T$ is identified with the self-map of
$[k] \times \widehat\Gamma_n$ that sends the $i$-th copy of $\widehat \Gamma_n$ to the $j$-th copy of $\widehat \Gamma_n$ by left-multiplication with $\gamma \in \Gamma_n$ for the unique $j \in [k]$, $\gamma \in \Gamma_n$ with $h_ig_i = g_j \gamma$.


As this is in particular true for $T\in\Gamma$, we obtain a chain of inclusions
\[
 \Gamma \leq \Gamma_n^\vee:= (\Gamma_n)^k \rtimes {\rm Sym}(k) \leq G=[\Gamma \curvearrowright (\widehat{\Gamma},\mu_H)],
\]
where the group in the middle is the permutational wreath product. In particular, $\Gamma$ is contained in a discrete subgroup which contains elements whose support has measure $1/k$.

The sequence of wreath products $(\Gamma_n^\vee)_n$ is an increasing sequence of subgroups of $G$ containing $\Gamma$, whose union is dense in $G$: indeed, for $T\in G$ and $\eps > 0$ arbitrary we can find $n$ large enough and $T'\in G$ such that $d(T,T')<\eps$ and $T'$ decomposes over
$\widehat \Gamma = \bigsqcup_{i=1}^k g_i \widehat\Gamma_n$. By the analysis above $T'$ belongs to $(\Gamma_n)^k \rtimes {\rm Sym}(k)$. Let's sumarize the result in the following proposition:

\begin{proposition} \label{dense}
Let $\Gamma$ be a countable residually finite group and consider the action on its profinite completion $\Gamma \curvearrowright (\widehat{\Gamma},\mu_H)$. There exists an increasing chain of discrete subgroups $\Gamma:= \Gamma_0 \leq \Gamma_1 \leq \cdots \leq G:= [\Gamma \curvearrowright (\widehat{\Gamma},\mu_H)]$, whose union is dense.
\end{proposition}

\begin{remark} \label{jacobs}
Let us now come back to the example of Jacobson mentioned already in the introduction. Jacobson showed that there exists a group $\Gamma$, which is elementary amenable and MIF, see \cite{MR4193626}. It was shown in \cite{MR2507115} that this group is also residually finite. Thus, we obtain that the full group $G$ of the unique hyperfinite equivalence relation contains two $1$-discrete copies of $\Gamma$, one is contained in a maximal discrete subgroup (by Theorem \ref{main1} applied to the Bernoulli action) and the other contained in an infinite chain of discrete overgroups, whose union is dense (by Proposition \ref{dense} applied to the action on the profinite completion).
\end{remark}

\subsection{Compact actions}

This section contains the proof of Theorem \ref{thm:compact}. By assumption, the infimum
\[
t:= \inf \{ \mu (S(g)) : g\in \Gamma\setminus \{ e \}  \}
\]
is strictly greater than $0$. Our goal is to show that $t=1$.

In what follows we will identify $\Gamma$ with its image in $\mathrm{Aut}(X,\mu )$. Let $K$ denote the closure of $\Gamma$ in $\mathrm{Aut}(X,\mu )$, which is compact by assumption.  We begin with the following claim.

\begin{claim}\label{claim:largeorder} For every $g_0\in \Gamma \setminus \{ 1 \}$, $\epsilon >0$, and natural number $n$, there exists some $g_1\in \Gamma$ of order greater than $n$ such that $\mu (S(g_1)\setminus S(g_0))<\epsilon$.
\end{claim}

\begin{proof}[Proof of Claim \ref{claim:largeorder}]
Let $V$ be an open identity neighborhood in $K$ satisfying $$\mu (kS(g_0)\triangle S(g_0))<\epsilon $$ for all $k\in V$. Then there must be some $k\in V\cap \Gamma$ such that $[k,g_0]$ has order greater than $n$. For suppose otherwise. Since $K$ is compact and $\Gamma$ is dense in $K$ there is some finite $F\subseteq \Gamma$ such that $VF=K$. Then for every $h\in \Gamma$ there is some $s\in F$ with $hs^{-1}\in V\cap \Gamma$ and hence some $1\leq i \leq n$ such that $[hs^{-1},g_0]^i = e$, so $\Gamma$ satisfies a nontrivial mixed identity by Proposition \ref{prop:mif}, a contradiction.

Let $k\in V\cap \Gamma$ be such that $g_1:= [k,g_0]$ has order greater than $n$. Then $S(g_1)$ is contained in $kS(g_0)\cup S(g_0)$, hence $\mu (S(g_1)\setminus S(g_0))\leq \mu (kS(g_0)\setminus S(g_0))<\epsilon$. This finishes the proof of Claim \ref{claim:largeorder}.
\end{proof}

Fix $1\geq \epsilon >0$, and let $g_0\in \Gamma\setminus \{ 1 \}$ be such that $\mu (S(g_0))<t + \epsilon$. Since $K$ is compact we can find a $K$-conjugation-invariant identity neighborhood $U$ in $K$ such that $\mu (kS(g_0)\triangle S(g_0))<\epsilon$ for all $k\in U$. By compactness again there is a natural number $n\geq 1$ such that for every $k\in K$ there is some $1\leq i< n$ with $k^i\in U$. By Claim \ref{claim:largeorder} there exists some $g_1\in \Gamma$ of order greater than $n$ such that $\mu (S(g_1)\setminus S(g_0))<\epsilon$. Let $1\leq i< n$ be such that $g_1^i \in U$ and let $g:=g_1^i$. We have $S(g)\subseteq S(g_1)$, and since $g\neq e$ we have $t\leq \mu (S(g))$, and therefore $\mu (S(g)\triangle S(g_0))< 2\epsilon$. It follows that for all $k\in U$ we have $\mu (kS(g)\triangle S(g))<5\epsilon$.

By ergodicity there exists a syndetic subset $D$ of $\Gamma$ such that for all $k\in D$ we have
\[
\mu (S(g)\cap kS(g)) < \mu (S(g))^2+\epsilon < t^2 + 9\epsilon .
\]
Then, arguing as in Claim \ref{claim:MIFnoncommute}, there exists some $k\in D$ such that $g$ and $kgk^{-1}$ do not commute. Let $h := kgk^{-1}$. Then $S(h)=kS(g)$, so $\mu (S(h))=\mu (S(g))$ and $\mu (S(g)\cap S(h) ) < t^2 +9\epsilon$. Since $U$ is conjugation invariant and $g$ belongs to $U$, both of the conjugates $kgk^{-1}$ and $k^{-1}gk$ belong to $U$ as well, hence
\begin{align*}
\mu (hS(g)\triangle S(g)) &= \mu (kgk^{-1}S(g)\triangle S(g) )<5\epsilon , \\
\text{and } \ \mu (gS(h)\triangle S(h)) &= \mu (k^{-1}gkS(g)\triangle S(g))<5\epsilon .
\end{align*}
Letting $A:=S(g)\cap S(h)$, it follows that $\mu (g A \cap A ) \geq \mu (A) - 5\epsilon$ and $\mu (hA\cap A)\geq \mu (A)-5\epsilon$. Since $g$ and $h$ do not commute we have $t\leq \mu (S([g,h]))$, so applying Proposition \ref{prop:subset} we obtain
\begin{align*}
t\leq \mu (S([g,h])) &\leq 3\mu (A) - \mu (gA\cap A)-\mu (hA\cap A) \\
&\leq \mu (A) + 10\epsilon \\
&\leq  t^2 + 19\epsilon .
\end{align*}
Since $\epsilon >0$ was arbitrary this shows that $t\leq t^2$, hence $t=1$. This finishes the proof of Theorem \ref{thm:compact}.

\begin{corollary}
Let $\Gamma=\Gamma_0\geq \Gamma_1\geq \Gamma_2\cdots$ be a chain of finite index subgroups of a {\rm MIF} group $\Gamma$. For $g\in \Gamma$ let
\[
t_g := \lim _{n\rightarrow\infty} \frac{ | \{ x\in \Gamma/\Gamma_n \, : \, gx = x \} |}{[\Gamma:\Gamma_n]} .
\]
Suppose that there is some nonidentity element $h\in \Gamma$ for which $t_h >0$. Then for any $\epsilon >0$ there exists some nonidentity element $g\in \Gamma$ such that $t_g>1-\epsilon$.
\end{corollary}

Thus, in terms of residual chains of finite index subgroups, we arrive at a dichotomy: Either the chain $(\Gamma_n)_n$ is a Farber chain (see \cite[Theorem 0.3]{MR1625742}), i.e. $t_g=0$ for all $g\neq 1$, or the opposite holds: for any $\varepsilon>0$, there exists $g \in \Gamma$, such that the probability that $g$ is contained in a random conjugate of $\Gamma_n$ is a least than $1-\varepsilon.$

Theorem \ref{thm:compact} implies that for every nonfree ergodic compact action of an MIF group $\Gamma$ on a probability space $(X,\mu )$, there exists a sequence $(g_n)_{n\geq 0}$ of nonidentity elements of $\Gamma$ satisfying $\mu (F(g_n))\rightarrow 1$. 
It is unclear whether this sequence $(g_n)_{n\geq 0}$ can be chosen independently of the nonfree ergodic compact action of $\Gamma$, even in the case where $\Gamma$ is a free group. Let us record this as a question:

\begin{question} \label{convseq}
Let $\Gamma$ be a nonabelian free group. Does there exist a sequence $(g_n)_{n\geq 0}$ of nonidentity elements of $\Gamma$ such that for every ergodic compact p.m.p.\ action $\Gamma\curvearrowright (X,\mu )$ which is not essentially free we have $\mu (\{ x: g_nx=x \} )\rightarrow 1$ as $n\rightarrow \infty$?
\end{question}

This should be compared to the results in \cite{MR3043070}. By results from that paper, there exists a sequence $(g_n)_{n \geq 0}$ in every non-abelian free group, such that $g_n \to 1$ in the weak topology for every compact action (no matter if the action is essentially free or not). A potential strategy to arrive at a positive answer to Question \ref{convseq} is to enumerate all non-trivial elements of $F_2$ in a sequence $(h_n)_{n \geq 0}$ and consider a sequence $(g_n)_{n \geq 0}$ as above. One could then start taking commutators of conjugates in some determined iteration scheme (compare to the proof of the main result in \cite{MR3043070}).

\section{A case study building on work of Choi and Blackadar}


In this section, we analyze a natural example of a free subgroup of the full group of a hyperfinite equivalence relation provided by a combination of the work of Choi \cite{MR540914} and Blackadar \cite{MR808296}.

Choi proved that the following unitaries in $M_2(\mathcal O_2)\cong \mathcal O_2$ satisfy $\widehat U^2 = 1$, $\widehat V^3 = 1$ and generate a copy of $C^*_r(\PSL(2,\Z))$ \cite{MR540914}:
\[
U = \begin{pmatrix} 0 & 1 \\ 1 & 0 \end{pmatrix},\quad
V = \begin{pmatrix} 0 & S_2^* \\ S_1 & S_2 S_1^*\end{pmatrix},
\]
where $S_1$ and $S_2$ are canonical isometries generating $\mathcal O_2$.

Blackadar, see  \cite{MR808296}, then produced an explicit $C^*$-subalgebra of the UHF algebra $\bigotimes_{k\in\Z} M_2(\C)$ which surjects onto $\mathcal O_2$.
In what follows, we describe his construction with slight change of notation.

Consider the crossed product $B=\left(\bigotimes_{k\in\mathbb Z} M_2(\mathbb C)\right)\rtimes \mathbb Z$, where $\mathbb Z$ acts by shifting the tensor factors. It is generated by the canonical unitary $z$ implementing the shift and a copy of $M_2(\mathbb C)$ in the $0$-th entry. We let the latter be generated by a projection $e_0$ and a unitary $t_0$ which maps $e_0$ to its complement: $t_0 e_0 t_0^* = 1-e_0$. Explicitly, we can take
\[
e_0=\begin{pmatrix} 1 & 0 \\ 0 & 0 \end{pmatrix},\quad t_0 = \begin{pmatrix} 0 & 1 \\ 1 & 0 \end{pmatrix} \in M_2(\mathbb C)
\]

The shifted copies of $e_0$ and $t_0$ will be denoted by $e_k= z^k e_0 z^{-k}$ resp. $t_k=z^k t_0 z^{-k}$, $k\in \mathbb Z$. We let $e\coloneqq e_{0}$. 

Blackadar then introduces the following elements:
\[
s_1\coloneqq z(1-e)=(1-e_1)z,\quad s_2 = t_1s_1 = t_1z(1-e) = zt_0(1-e) = t_1(1-e_1)z.
\]
They satisfy
\[
s_1^*s_1 = s_2^*s_2 = 1-e,\quad s_1s_1^* = 1-e_1, \quad s_2s_2^* = e_1
\]
and therefore their images $S_1$ and $S_2$  in the quotient by a suitable ideal containing $e$ generate the Cuntz algebra $\mathcal O_2$.
\begin{lemma}
The following elements are unitary lifts of $U$ and $V$ into $M_2(B)$:
\[
\widehat U = \begin{pmatrix} 0 & 1 \\ 1 & 0 \end{pmatrix},\quad
\widehat V = \begin{pmatrix} e & s_2^* \\ s_1 & s_2 s_1^*\end{pmatrix}
\]
which generate a copy of $\PSL(2,\mathbb Z)$ inside ${\rm U}(M_2(B))$.
\end{lemma}
\begin{proof}
Indeed,
\[
\widehat V^* = \begin{pmatrix} e & s_1^* \\ s_2 & s_1 s_2^*\end{pmatrix},
\]
and so
\[
\widehat V^* \widehat V = \begin{pmatrix} e+s_1^*s_1 & es_2^* + s_1s_2^*s_1 \\ s_2e+s_1^*s_2s_1^* & s_2s_2^* + s_1s_2^*s_2s_1^* \end{pmatrix} = \begin{pmatrix} 1 & 0 \\ 0 & 1 \end{pmatrix},
\]
since $s_1^*s_1 = 1-e$, $s_2^*s_1 = 0$, $s_2 e = 0$ and
\[
s_2s_2^* + s_1s_2^*s_2s_1^* = e_1 + s_1(1-e)s_1^* = 1 - s_1es_1^* = 1.
\]

Moreover,
\[
\widehat V^2 = \begin{pmatrix} e+s_2^*s_1 & es_2^* + s_2^*s_2s_1^* \\ s_1e+s_2s_1^*s_1 & s_1s_2^* + s_2s_1^*s_2s_1^* \end{pmatrix} = \widehat V^*.
\]

This finishes the proof.
\end{proof}

We equip $B$ with the canonical normalized trace $\tau$ coming from the crossed product structure using the standard tensor product trace on $\bigotimes_{k\in \mathbb Z} M_2(\mathbb C)$. We then canonically extend $\tau$ to a normalized trace on $M_2(B)$.

In view of the existence of the surjection of the subalgebra generated by $U$ and $V$ onto $C^*_r({\rm PSL}_2(\mathbb Z))$, this representation of ${\rm PSL}_2(\mathbb Z)$ into ${\rm U}(M_{2^{\infty}}(\mathbb C))$ is ``as non-amenable as it gets'' and provides a promising candidate of a free copy of ${\rm PSL}_2(\mathbb Z)$ (and hence $F_2$) in ${\rm U}(R)$ with the $2$-norm.

In the next step, we observe that this copy of ${\rm PSL}_2(\mathbb Z)$ actually sits inside the full group of a natural hyperfinite equivalence relation. Indeed, consider the Cantor space $X=\{0,1\}^\Z\times \{0,1\}$ with the natural product measure giving each bit weight $1/2$. We will interpret it as the state space of a Turing machine with a bi-infinite tape with zeroes and ones and an additional state (or signal) that can take values in $\mathbb Z/2$. We identify the algebra $C(X)$ with the diagonal subalgebra in $\bigotimes_{k\in \Z} M_2(\C)$; the trace $\tau$ then corresponds to the aforementioned product measure. Now, the unitary $z$ corresponds to the tape shift, and $t_0$ corresponds to switching the signal. Therefore the unitaries $\widehat U$ and $\widehat V$ are elements of the full group $[\mathcal R]$ of the hyperfinite equivalence relation given by the natural measure preserving action of $(\Z/2\Z \wr \mathbb Z)\times \Z/2\Z$ on $X$, where the first factor acts by shifting and changing entries on the tape, and the second factor acting by changing the signal. It is routine to check that under this identification the trace $\tau$ of an element $g\in[\mathcal R]$ is equal to the measure of the set of the fixed points of $g$.

To make the computations easier, we will conjugate $\widehat U$ and $\widehat V^*$ to the elements
\[
u= \diag(1,z^*)\cdot \widehat U\cdot \diag(1,z)= \begin{pmatrix} 0 & z \\ z^* & 0 \end{pmatrix}
\]
and
\[
v= \diag(1,^*) \widehat V^* \diag(1,z)= \begin{pmatrix} e & s_1^*z \\ z^*s_2 & z^*s_1 s_2^*z\end{pmatrix} = \diag(1,t_0)\cdot\begin{pmatrix} e & 1-e \\ 1-e & e \end{pmatrix}.
\]

Now, $u$ has the interpretation that it changes the signal and shifts the tape according to the signal to the right (if the signal was $1$) or to the left (if the signal was $0$). In particular, it does not change the tape -- it only changes the signal and shifts the tape. The element $v$ has the interpretation that it changes the signal iff we read a $1$ at the zeroeth entry and afterwards, if the signal is $1$, it also changes the zeroeth entry of the tape. In particular, the operation $v$ does not shift the tape.

We summarize the above observations as follows:
\begin{proposition}
Consider the Cantor space $X=\{0,1\}^\Z \times \{ 0,1\}$ with the natural product measure $\mu$ assigning weight $1/2$ to every bit, understood as the state space of a Turing machine whose tape is bi-infinite with entries from the alphabet $\{0,1\}$ and whose only internal state, the signal $S$, is taking values in $\mathbb Z/2$. Let $x_0$ denote the 0-th entry on the tape.

Consider the following commands of this Turing machine:
\begin{itemize}
\item[$u$:] if $S = 0$, then shift the tape to the left and set $S\coloneqq 1$; if $S = 1$, shift the tape to the right and set $S\coloneqq 0$;

\item[$v$:] permute the pairs $(x_0,S)$ as follows: $(0,0)\mapsto (0,0)$, $(0,1)\mapsto(1,1)\mapsto(1,0)\mapsto(0,1)$ without shifting the tape.
\end{itemize}

Then $u$ and $v$ act on $X$ as probability measure-preserving automorphisms of order $2$ and $3$ respectively, generating a copy of $\PSL(2,\Z)$ inside the full group of a hyperfinite equivalence relation on $X$.
\end{proposition}

Since Theorem \ref{main2} is only true for ergodic actions, we have to investigate the ergodicity of the above action. It turns out that it is not ergodic as such, but has a transparent description of the ergodic components. 

We let $\pi_0\colon \{0,1\}^{\Z}\to \{0,1\}^{-\N}$ be the natural projection discarding everything on the right from the 0-th entry and $\pi_{-1}\colon \{0,1\}^{\Z}\to \{0,1\}^{-\N}$ be the natural projection discarding everything on the right from the $(-1)$-st entry. We furthermore let $L\colon \{0,1\}^{-\N}\setminus\{(\dots,1,1,1)\}\to \{0,1\}^{-\N}$ be the map that discards the rightmost zero in the sequence and all entries on the right from it.

\begin{proposition}\label{prop:isomorphic-ergodic-components}
The following map
\[
p\colon (X,\mu)\to (\{0,1\}^{-\N},p_*\mu),
\]
\[
p(x,0) = (L\circ \pi_0)(x),\quad p(x,1) = (L\circ \pi_{-1})(x).
\]
is the ergodic decomposition of the action $\PSL(2,\Z)\curvearrowright (X,\mu)$. All ergodic components of this action are isomorphic; in particular, for each $g\in \PSL(2,\Z)$ the measure of the fixed point set is a.e. constant on the space of ergodic components.
\end{proposition}
\begin{proof}
The key observation here is the following: if $(x,S)$ is such that $(x_0,S) = (0,0)$, then $v$ does not change the pair $(x_0,S)$, and then $u$ is the only nontrivial operation, which necessarily shifts the tape to the left. Furthermore, in an arbitrary state $u$ is the only operation which can shift the tape and after shifting it to the right the signal $S$ is set to $0$. Therefore, if $x_{-1} = 0$ and we shift the tape to the right (by applying $u$), then we necessarily get $(x_0,S) = (0,0)$, from which we can only shift the tape to the left.

Together, this implies the following claim: given an initial state $(x,S)$, we can never shift the tape beyond the right-most zero in $\pi_0(x)$ if $S=0$ or beyond the right-most zero in $\pi_{-1}(x)$ if $S=1$ (we thus refer to this zero on the tape as the ``stopping zero''). Now, the map $p$ is exactly the map which discards the stopping zero and the half-tape to the right of it (ignoring the null set of states where no zero occurs on the on the strictly negative half of the tape).

By the above considerations, each fiber $p^{-1}(y)$ is invariant under the action, being exactly the set of states which have the prefix $y$ to the left of the stopping zero. By disintegration of measures, it comes naturally equipped with the probability measure $\nu_y$ which can be described as follows. Consider the product measure $\theta$ on $\{0, 1\} ^\N$ and the maps 
\[
\varphi_y \colon \{ 0, 1\} ^\N \setminus \{ (1,1,1,\dots ) \} \to
\{0,1\}^{\Z}\times\{0\},
\]
\[
1^k0x \mapsto (y01^kx, 0),
\]
where the leftmost bit of $x$ occupies coordinate $1$ in the string $y01^kx$, and
\[
\psi_y \colon \{ 0, 1\} ^\N \setminus \{ (1,1,1,\dots ) \} \to
\{0,1\}^{\Z}\times\{1\},
\]
\[
1^k0x \mapsto (y01^kx, 1),
\]
where the leftmost bit of $x$ occupies coordinate $0$ in the string $y01^kx$. Then $\nu_y = \frac12(\varphi_y)_*\theta + \frac12(\psi_y)_*\theta$. Equivalently, $\nu _y =\tfrac{1}{2}\nu _y ^0 + \tfrac{1}{2}u_*\nu _y^0$, where $\nu _y^0 = (\varphi_y)_*\theta$.

Let us check that the action of $\PSL(2,\Z)$ on $(p^{-1}(y),\nu_y)$ is ergodic. We first prove the following claim: given an arbitrary initial state of the form $(z,S)=(y01^jtx,S)\in p^{-1}(y)$ (where $x$ is the infinite tail and $t$ is a finite string of length $k$ whose leftmost bit occupies coordinate $1$ if $S=0$, and coordinate $0$ if $S=1$) together with an arbitrary $S'\in \{0,1\}$, there is an element $g\in \PSL(2,\Z)$ such that $g(z,S)=(y01^{j+k}x,S')$, where each bit of $t$ is replaced by a $1$ and the leftmost bit of $x$ occupies coordinate $1$. If $S=0$, we first apply $u$ to shift the tape to the left. Then we repeat the following procedure $k$ times: apply a power of $v$ to get $(z_0,S) = (1,0)$ (this is possible because now $S = 1$), then apply $u$ again to shift the tape further to the left. Finally, we can apply $v$ if necessary to change $S$ to $S'$, retaining $1$ on the tape, finishing the proof of the claim. 

Let $E_y$ denote the orbit equivalence relation of the action of $\PSL(2,\Z )$ on $p^{-1}(y)$, and let $E_y^0$ denote its restriction to $p^{-1}(y)\cap \{ 0, 1 \}^{\Z} \times \{ 0 \}$.  It follows from the claim that the map $\varphi _y$ defined above gives an isomorphism from the equivalence relation of eventual equality on $\{ 0, 1 \} ^{\N}$ (equipped with product measure) to $E_y^0$; to see that $\varphi _y$ gives an isomorphism with $E_y^0$ and not just with one of its subequivalence relations, observe that the map $p^{-1}(y)\rightarrow \{ 0, 1\} ^{\N}$, given by $(z,0)\mapsto \varphi _y^{-1}(z,0)$ and $(z,1)\mapsto \psi _y ^{-1}(z,1)$, maps $E_y$-equivalent points to eventually equal sequences. Since $\nu _y =\tfrac{1}{2}\nu _y ^0 + \tfrac{1}{2}u_*\nu _y^0$, this implies that the action of $\PSL (2,\Z )$ on $(p^{-1}(y),\nu _y )$ is ergodic.

Finally, it is easy to see that the actions of $\PSL(2,\Z)$ on arbitrary two fibers $p^{-1}(y)$ and $p^{-1}(y')$ are isomorphic through the obvious map which replaces the prefix $y$ to the left of the stopping zero with the prefix $y'$.
\end{proof}

We thus are interested in understanding the measures of fixed point sets of words in $u$ and $v$. The above proposition suggests an efficient method to evaluate them on a computer. Indeed, it is obvious that for every initial state of the Turing machine a word $w=w(u,v)$ in $u$ and $v$ can shift the tape at most by the number of occurences of $u$ in $w$ which we denote by $|w|_u$. Therefore, it is enough to evaluate the word $w$ on $2^{2|w|_u + 2}$ possible initial configurations of the Turing machine ($2|w|_u + 1$ bits on the tape and 1 bit of the signal), checking whether the Turing machine returns to the initial configuration; the proportion of the fixed configurations is exactly the measure of the fixed point set.

Moreover, it turns out that on most of the initial configurations the tape will actually be shifted by far less than $|w|_u$ in one direction. Thus, to estimate the measure of the fixed points from below, it is enough to bound the tape displacement by some number $\ell$, discarding an initial configuration as ``possibly non-fixed'' once the tape displacement exceeds $\ell$ in the process of applying $w(u,v)$. This reduces the number of initial configurations to be checked to $2^{2\ell + 2}$.

Using computer search and the idea of applying iterated commutators, we were able to identify following elements with corresponding lower trace estimates with displacement bound $\ell \coloneqq 7$ (see the \textsc{Magma} code in the Appendix). We let $a\coloneqq uv$, $b\coloneqq uv^2$,
\[
g_1\coloneqq [a^{14} v a^{-14},v], \quad \tau(g_1)\geq 0.53,
\]
\[
g_2\coloneqq [a^9 g_1a^{-9},g_1],\quad \tau(g_2)\geq 0.64,
\]
\[
g_3\coloneqq [b^2g_2b^{-2},g_2], \quad \tau(g_3)\geq 0.69.
\]

One quick way to see that ${\rm PSL}_2(\mathbb Z)$ is MIF is to note that the shortest mixed identity of ${\rm PSL}_2(p)$ is of length $\Omega(p)$, see \cite{bst}. Applying Proposition \ref{prop:isomorphic-ergodic-components} and Theorem \ref{main2}, we thus obtain the following:

\begin{theorem}
The above copy of ${\rm PSL}_2(\mathbb Z)$ is not discrete as a subgroup of $([\mathcal R],d)$ resp. $({\rm U}(R),\lVert{\cdot}\rVert_2)$.
\end{theorem}

We include this example so that it becomes obvious that Theorem \ref{main2} is actually quite useful when studying concrete examples. We take that example as evidence that the answer to Question \ref{q2} might be negative.

\appendix
\section{Computation of the trace}
This is the source code of a \textsc{Magma} program giving the estimates for the elements in Section 5. It can be sucessfully executed on the free \textsc{Magma} online calculator \url{http://magma.maths.usyd.edu.au/calc/} if you restrict the computation to the first element (see the end of the code).

\vspace*{1em}
\begin{scriptsize}
\begin{verbatim}
QQ:=RationalField();
RR:=RealField(4);
ZZ:=Integers();

tape0:=AssociativeArray(ZZ);
signal0:=0;
state0:=[*signal0,tape0,0*];

function u(state,init)
    res:=state;
    if state[1] eq 0 then res[1]:=1; res[3]:=res[3]+1; end if;
    if state[1] eq 1 then res[1]:=0; res[3]:=res[3]-1; end if;
    return res,init;
end function;

function v(state,init)
    res:=state;
    rinit:=init;
    if res[1] eq 0 and res[2][state[3]] eq 0 then return res,rinit; end if;
    if res[1] eq 0 and res[2][state[3]] eq 1 then res[1]:=1; res[2][state[3]]:=0;
    return res,rinit; end if;
    if res[1] eq 1 and res[2][state[3]] eq 0 then res[1]:=1; res[2][state[3]]:=1;
    return res,rinit; end if;
    if res[1] eq 1 and res[2][state[3]] eq 1 then res[1]:=0; res[2][state[3]]:=1;
    return res,rinit; end if;
end function;

function vv(state,init)
    res:=state;
    rinit:=init;
    if res[1] eq 0 and res[2][state[3]] eq 0 then return res,rinit; end if;
    if res[1] eq 0 and res[2][state[3]] eq 1 then res[1]:=1; res[2][state[3]]:=1;
    return res,rinit; end if;
    if res[1] eq 1 and res[2][state[3]] eq 1 then res[1]:=1; res[2][state[3]]:=0;
    return res,rinit; end if;
    if res[1] eq 1 and res[2][state[3]] eq 0 then res[1]:=0; res[2][state[3]]:=1;
    return res,rinit; end if;
end function;


function iseq(state,init)
    if state[3] ne 0 then return false; end if;
    if state[1] ne init[1] then return false; end if;
    if #Keys(init[2]) ne #Keys(state[2]) then return false; end if;
    for k in Keys(init[2]) do
        if init[2][k] ne state[2][k] then return false; end if;
    end for;
    return true;
end function;

function apply_generator(gen,state,init)
    if gen eq 1 or gen eq -1 then return u(state,init); end if;
    if gen eq 2 then return v(state,init); end if;
    if gen eq -2 then return vv(state,init); end if;
end function;

function precise_trace(h,prec)
    tr:=0;
    badness:=0;
    state:=[*0,AssociativeArray(ZZ),0*];    
    bits:=[];
    pwrs:= ElementToSequence(h);
    for k in [0..2^(2*prec+2)-1] do
        if (k mod 2^(2*prec-2)) eq 0 then
            printf "prc progress = %o trace = %o badness = %o\n",
            RR! k/2^(2*prec+2), RR ! tr, RR ! badness;
            end if;
        bits:=IntegerToSequence(k,2);
        for i in [#bits+1..2*prec+2]  do Append(~bits,0); end for;
        state[1]:=bits[1];
        for l in [-prec..prec] do state[2][l]:=bits[prec+l+2]; end for;
        state[3]:=0;
        init:=[*state[1],state[2]*];
        for i in [1..#pwrs] do 
            state,init:=apply_generator(pwrs[i],state,init);
            if state[3] gt prec or state[3] lt -prec then
                badness:=badness+1/2^(2*prec+2);
                continue k;
            end if;
        end for;
        if iseq(state,init) then tr:=tr+1/2^(2*prec+2); end if;    
    end for;
    return tr,badness;
end function;

G<u,v>:=FPGroup<u,v | u^2 = v^3 = 1>;

a:=v*u;
b:=v^2*u;
comm:=[];
comm[1]:=(a^14*v*a^-14,v); //~0.53
print "prc",RR ! precise_trace(comm[1],7);

//Comment the remaining lines out if you want to run the code
//on free Magma online calculator http://magma.maths.usyd.edu.au/calc/

comm[2]:=(a^9*comm[1]*a^-9,comm[1]); //~0.64
print "prc",RR ! precise_trace(comm[2],7);
comm[3]:=(b^2*comm[2]*b^-2,comm[2]); //~0.69
print "prc",RR ! precise_trace(comm[3],7);
\end{verbatim}
\end{scriptsize}

\section*{Acknowledgments}

A manuscript written by the third-named author containing the ideas outlined in the introduction circulated in 2015. The first-named and the third-named author acknowledge funding by the Deutsche Forschungsgemeinschaft (SPP 2026). The fourth-named author was supported in part by NSF grant DMS 2246684.


\end{document}